\documentclass[11pt]{article}
\usepackage[margin=1.2in]{geometry}
\usepackage{amsmath,amssymb,amsthm}
\usepackage[T1]{fontenc}

\theoremstyle{plain}
\newtheorem{theorem}{Theorem}[section]
\newtheorem{proposition}[theorem]{Proposition}
\newtheorem{lemma}[theorem]{Lemma}
\newtheorem{corollary}[theorem]{Corollary}
\newtheorem*{theoremA}{Theorem A}
\newtheorem*{theoremB}{Theorem B}
\theoremstyle{definition}
\newtheorem{definition}[theorem]{Definition}
\newtheorem{remark}[theorem]{Remark}
\newtheorem{example}[theorem]{Example}

\newcommand{\On}{\mathcal{O}_n}
\newcommand{\im}{\operatorname{im}}

\newcommand{\Rel}{\mathrel{R}}
\newcommand{\Rcl}{\mathcal{R}}
\newcommand{\Lcl}{\mathcal{L}}
\newcommand{\Dcl}{\mathcal{D}}
\newcommand{\Hcl}{\mathcal{H}}

\title{Balanced weightings and permutation matchings\\ applied to the monoid of order-preserving mappings}
\date{}
\author{Peter M. Higgins\\
School of Mathematics, Statistics, and Actuarial Science \\
University of Essex, Colchester CO4 3SQ U.K. \\
emails: peteh@essex.ac.uk \\
}

\begin{document}
\maketitle

\begin{abstract}
We give a necessary and sufficient condition for a combinatorial $\Dcl$-class of a finite semigroup to have a permutation matching: that its structure matrix admit a \emph{balanced weighting}, that is, non-negative weights with constant row sums and constant column sums. This weakens the regularity hypothesis of \cite[Lemma 2.9]{H14}, the role of Hall's lemma being taken up by the Birkhoff--von Neumann theorem. We then apply it to show that $\On$, the monoid of all order-preserving self-maps of an $n$-element chain, has a permutation matching, answering Question~1 of \cite{H24}. The weighting there counts the ways of deleting one entry from a multiset.
\end{abstract}

\section{Introduction}

We assume a general background in Green's relations as explained in the text \cite{Howie}. Following \cite{H14}, a \emph{(permutation) matching} of a finite regular semigroup $S$ is a permutation $\sigma$ of $S$ with $\sigma(x)$ an inverse of $x$ for all $x$; since all inverses of $x$ lie in $D_x$, the notion of matching also applies to $\mathcal{D}$-classes. Not every finite regular semigroup has a matching, the minimal failure being a seven-element $0$-rectangular band \cite[Example 1.3]{H14}; by \cite[Proposition 1.2]{H14} the obstruction is the Hall condition $|A| \le |V(A)|$ for all $A \subseteq S$, where $V(A)$ denotes the set of inverses of members of $A$, and by \cite[Theorem 1.6]{H14} the question reduces to that of $0$-rectangular bands $(D_x \cup \{0\})/\Hcl$.

The positive results of \cite{H14} rest on a regularity argument: a bipartite graph regular of degree $m$ satisfies Hall's condition, since $k$ vertices on one side carry $km$ edges into their neighbourhood, and this is what eventually yields a permutation matching for the full transformation monoid, $T_n$ \cite[Lemma 2.9 and Theorem 2.12]{H14}. Our first theorem is a general result that shows how regularity may be weakened to a weighted form of itself, at the cost of replacing Hall's lemma by Birkhoff's theorem.

Throughout we work with a fixed combinatorial regular $\mathcal{D}$-class $J$ of a finite semigroup, index its $\mathcal{R}$-classes by $P$ and its $\mathcal{L}$-classes by $Q$, identify $J$ with $P\times Q$, and define
\[
R \;=\; \{\,(A,B) \in P \times Q :  (A,B) \text{ is idempotent}\,\},
\]
noting that every $\mathcal{R}$-class and every $\mathcal{L}$-class of $J$
contains an idempotent, because $J$ is regular, and that $R$ may be identified with the binary matrix that is the transpose of the structure matrix, in the Rees sense, of the
$0$-rectangular band $J \cup \{0\}$. The distinction is immaterial here: a
function $w$ is a balanced weighting of $R$ with constants $(a,b)$ if and only
if its transpose is a balanced weighting of $R^{\mathsf T}$ with constants
$(b,a)$, and so the condition appearing in Theorem~A is unaffected by the choice of
orientation. We shall accordingly refer to $R$ itself as the structure matrix
of $J$.

\begin{theoremA}
Let $J$ be a combinatorial regular $\Dcl$-class of a finite semigroup, with structure matrix $R$. Then $J$ has a permutation matching if and only if $R$ admits a balanced weighting.
\end{theoremA}

A balanced weighting is a non-negative weighting of the cells of $R$ whose row sums are all equal and whose column sums are all equal; the $0$-$1$ case is a bi-regular bipartite graph (with the rows and columns representing the independent sets), and the constant case on a regular graph is the hypothesis of \cite[Lemma 2.9]{H14}. This new hypothesis is not automatic: Example~\ref{ex:fail} exhibits a structure matrix, with no empty row or column, whose $\Dcl$-class has no permutation matching.

Application of Theorem A answers the question that motivated this research. Whether $\On$, the (regular) monoid of all order-preserving maps on the chain $[n] = \{1 < 2 < \cdots < n\}$, has a permutation matching is Question~1 of \cite{H24}.

\begin{theoremB}
For every $n \ge 1$, the monoid $\On$ has a permutation matching.
\end{theoremB}

The argument of \cite{H14} for the full transformation semigroup applied Hall's lemma through exploiting the regularity of numbers of inverses across strips of a $\Dcl$-class consisting of $\Rcl$-classes with the same numerical partition.  However this uniformity breaks down when we restrict to isotone mappings. The model of argument used to show that $\mathcal{OP}_n$, the monoid of all orientation-preserving mappings, has an involution matching makes use of the natural duality between $\Lcl$- and $\Rcl$-classes to establish a canonical matching \cite[Example 2.6]{H14}. However the $\Dcl$-classes of $\On$ are not square but rather the ratio of $\Lcl$- to $\Rcl$-classes in $D_r$ (the $\mathcal {D}$-class of all mappings of rank $r$), is $n:r$. ($\On$ is a subsemigroup of $\mathcal{OP}_n$, but is not closed under the canonical inverse map of that example.)  What we show in \S\ref{sec:del} is that regularity is restored if the incidences are counted with a natural multiplicity: encoding kernels and images as multisets, the transversal relation becomes interlacing, and weighting a pair by the number of entries of the larger multiset whose deletion produces the smaller one gives a balanced weighting. Theorem~B then follows from Theorem~A.

\section{Crossed squares and balanced weightings}\label{sec:general}

Throughout this section $P$ and $Q$ are finite non-empty sets and $R \subseteq P \times Q$. For $A \in P$ and $B \in Q$ put
\[
N(A) = \{\, D \in Q : (A,D) \in R \,\}, \qquad M(B) = \{\, C \in P : (C,B) \in R \,\},
\]
and for $X \subseteq P$ write $R(X) = \bigcup_{A \in X} N(A)$. We regard $R$ as a $0$-$1$ matrix with rows indexed by $P$ and columns by $Q$,
so that the row of $R$ at $A$ encodes $N(A)$ and its column at $B$ encodes $M(B)$; in
particular $R$ has no empty row or column precisely when every $N(A)$ and every
$M(B)$ is non-empty. In the setting of \S1, $N(A)$ is non-empty exactly when the
$\mathcal{R}$-class $A$ contains an idempotent, and $M(B)$ exactly when the
$\mathcal{L}$-class $B$ does.

We say that a permutation $\sigma$ of a finite set $V$ is \emph{contained in} a relation $N \subseteq V \times V$ if $(v,\sigma(v)) \in N$ for every $v \in V$. The purpose of the next definition is to capture the relationship between pairs of inverses and pairs of idempotents in a semigroup $\mathcal D$-class. 

\begin{definition}\label{def:crossed}
The \emph{crossed square} of $R$ is the relation $\Gamma(R)$ on $P \times Q$ given by
\[
\Gamma(R) \;=\; \bigl\{\, \bigl((A,B),(C,D)\bigr) \;:\; (A,D) \in R \ \text{ and } \ (C,B) \in R \,\bigr\} .
\]
\end{definition}

Interchanging the two arguments leaves the defining condition unaltered, so $\Gamma(R)$ is symmetric. The name records that the two coordinates are constrained separately, but each against the other element's coordinate. Equivalently, the elements of $P \times Q$ related to $(A,B)$ in $\Gamma(R)$ are
precisely those of the product $M(B) \times N(A)$, the roles of $A$ and $B$
being exchanged.

\begin{definition}\label{def:weighting}
A \emph{balanced weighting} of $R$ is a function $w : P \times Q \to [0,\infty)$ vanishing off $R$ for which there are constants $a,b>0$ with
\[
\sum_{D \in Q} w(A,D) = a \quad (A \in P), \qquad \sum_{A \in P} w(A,D) = b \quad (D \in Q).
\]
\end{definition}

Summing all weights in the two possible orders gives $a|P| = b|Q|$. If $w$ takes only the values $0$ and $1$ then $R$, read as a bipartite graph, is bi-regular of degrees $(a,b)$; if in addition $a = b = m$ it is regular of degree $m$, which is the hypothesis of \cite[Lemma 2.9]{H14}.

A square matrix is \emph{doubly stochastic} if its entries are non-negative and every row sum and every column sum equals $1$, and a \emph{permutation matrix} is a doubly stochastic matrix with entries $0$ and $1$. We use the classical theorem of Birkhoff in the following form.

\begin{theorem}[Birkhoff \cite{B46}]\label{thm:birkhoff}
Every doubly stochastic matrix is a convex linear combination of permutation matrices: if $M$ is doubly stochastic then $M = \sum_{k} \lambda_k P_k$ for some permutation matrices $P_k$ and reals $\lambda_k > 0$ with $\sum_k \lambda_k = 1$.
\end{theorem}

\begin{lemma}\label{lem:support}
Let $V$ be a finite set and $N \subseteq V \times V$. If there is a doubly stochastic matrix $M$, with rows and columns indexed by $V$, whose entry $M(u,v)$ is non-zero only when $(u,v) \in N$, then some permutation of $V$ is contained in $N$.
\end{lemma}

\begin{proof}
Write $M = \sum_k \lambda_k P_k$ as in Theorem~\ref{thm:birkhoff}, fix any $k$, and let $\sigma$ be the permutation with matrix $P_k$. For each $u$ the entry of $P_k$ at $(u,\sigma(u))$ is $1$, and all terms of the sum are non-negative, so $M(u,\sigma(u)) \ge \lambda_k > 0$. Hence $(u,\sigma(u)) \in N$ for every $u$. 
\end{proof}

\begin{proposition}\label{prop:crossed}
If $R$ admits a balanced weighting, then some permutation of $P \times Q$ is contained in $\Gamma(R)$.
\end{proposition}

\begin{proof}
Let $w$ be a balanced weighting with constants $a,b$ and set, for $u = (A,B)$ and $v = (C,D)$,
\[
M(u,v): \;=\; \frac{w(A,D)\,w(C,B)}{ab} \;\ge\; 0 .
\]
The sum over the row indexed by $u$ factorises because the summation indices $C$ and $D$ occur in separate factors, and so:
\[
\sum_{(C,D)} M(u,v) \;=\; \frac{1}{ab}\Bigl(\sum_{D} w(A,D)\Bigr)\Bigl(\sum_{C} w(C,B)\Bigr) \;=\; \frac{a\,b}{ab} \;=\; 1 ,
\]
by the row condition of Definition~\ref{def:weighting} for the sum over $D$ and the column condition for the sum over $C$. Dually each column sum is $\tfrac{1}{ab}\cdot b\cdot a = 1$, so $M$ is doubly stochastic. If $M(u,v) > 0$ then $w(A,D)>0$ and $w(C,B)>0$, so $(A,D) \in R$ and $(C,B) \in R$, that is $(u,v) \in \Gamma(R)$. Applying Lemma~\ref{lem:support} now gives the result. 
\end{proof}

\begin{proposition}\label{prop:converse}
If some permutation of $P \times Q$ is contained in $\Gamma(R)$, then $R$ admits a balanced weighting, with constants $a = |Q|$ and $b = |P|$.
\end{proposition}

\begin{proof}
Let $\sigma$ be such a permutation and write $\sigma(A,B) = \bigl(C(A,B),\, D(A,B)\bigr)$. Define
\[
w(A,D) \;=\; \bigl|\{\, B \in Q \;:\; D(A,B) = D \,\}\bigr| .
\]
If $w(A,D) > 0$ then $D = D(A,B)$ for some $B$, and $\bigl((A,B),\sigma(A,B)\bigr) \in \Gamma(R)$ gives $(A,D) \in R$; so $w$ vanishes off $R$.

Fix $A \in P$. Every $B \in Q$ satisfies $D(A,B) = D$ for exactly one $D \in Q$, so the $|Q|$ elements $B$ are distributed among the values $w(A,D)$, whence $\sum_{D} w(A,D) = |Q|$.

Fix $D \in Q$. Then
\[
\sum_{A \in P} w(A,D) \;=\; \bigl|\{\, u \in P\times Q : D(u) = D \,\}\bigr| \;=\; \bigl|\sigma^{-1}\bigl(P \times \{D\}\bigr)\bigr| \;=\; \bigl|P \times \{D\}\bigr| \;=\; |P| ,
\]
the third equality because $\sigma$ is a bijection. Therefore $w$ is a balanced weighting with $a = |Q|$ and $b = |P|$. 
\end{proof}

Note that only the column condition uses that $\sigma$ is a bijection; the row condition
holds for any function whose graph lies in $\Gamma(R)$. The asymmetry is not an
artefact of the construction: in Example~2.9 such a function exists (every element of $P \times Q$ has a neighbour in $\Gamma(R)$, since $R$ has no empty row or column), but no
balanced weighting does.

\section*{Proof of Theorem A}

\begin{proof}

$\Rightarrow)$ For $x=(A,B) \in J$, $y=(C,D)$ is an inverse of $x$ if and only if $\bigl((A,B),(C,D)\bigr) \in \Gamma(R)$, \cite[Theorem 2.3.4]{Howie}. Hence the given permutation matching of $J$ is contained in  $\Gamma(R)$, and so by Proposition ~\ref{prop:converse}, $R$ admits a balanced weighting with constants $a = |Q|$ and $b = |P|$. 

$\Leftarrow)$ Conversely, given a balanced weighting of $R$, by Proposition \ref{prop:crossed}, $\Gamma (R)$ contains a  permutation, which is the required matching for $J$. 
\end{proof}

\begin{corollary}\label{cor:assemble}
A finite regular semigroup $S$ has a permutation matching if and only if the
structure matrix of each $0$-rectangular band $(D_x \cup \{0\})/\mathcal{H}$
admits a balanced weighting.
\end{corollary}

\begin{proof}
By \cite[Theorem 1.6]{H14}, $S$ has a permutation matching if and only if each $0$-rectangular band $(D_x \cup \{0\})/\mathcal{H}$ has a permutation matching, whence the result follows from Theorem A. 
\end{proof}

It is convenient to have a numerical criterion for the existence of a balanced
weighting. We use the max-flow min-cut theorem \cite{FF} in the following form. A \emph{network} consists of a finite directed graph with two distinguished vertices, a \emph{source} $\sigma$ and a \emph{sink} $\tau$, together with a capacity $c(\varepsilon) \in [0,\infty]$
for each arc $\varepsilon$. A \emph{flow} is an assignment $f(\varepsilon) \in [0,
c(\varepsilon)]$ to the arcs such that at every vertex other than $\sigma$ and $\tau$ the
total entering equals the total leaving; its \emph{value} is the net amount leaving
$\sigma$. A \emph{cut} is a set $S$ of vertices with $\sigma \in S$ and $\tau \notin S$, and its \emph{capacity} is the sum of $c(\varepsilon)$ over the arcs $\varepsilon$ running from a vertex of $S$ to a vertex outside $S$. The theorem asserts that, provided some cut has finite capacity, the maximum value of a flow equals the minimum capacity of a cut.

\begin{theorem}\label{thm:criterion}
Let $P$ and $Q$ be finite non-empty sets and $R \subseteq P \times Q$. Then $R$ admits a
balanced weighting if and only if
\begin{equation}\label{eq:crit}
|Q|\,|X| \;\le\; |P|\,|R(X)| \qquad \text{for every } X \subseteq P .
\end{equation}
\end{theorem}

\begin{proof}
Write $p = |P|$ and $q = |Q|$. We may normalise the constants of Definition 2.2: if $w$ is a balanced weighting with constants $a$ and $b$ then, since $ap = bq$, the weighting
$(q/a)\,w$ has row sums $q$ and column sums $p$. And so throughout we take
\begin{equation}\label{eq:norm}
\sum_{D \in Q} w(A,D) = q \quad (A \in P), \qquad
\sum_{A \in P} w(A,D) = p \quad (D \in Q),
\end{equation}
and note that the total weight is then $pq$, counted either way.

\smallskip
\noindent\emph{Necessity.} Let $w$ satisfy \eqref{eq:norm} and let $X \subseteq P$. Summing
the row conditions over $X$ gives $\sum_{A \in X}\sum_{D \in Q} w(A,D) = |X|\,q$. Every
non-zero term of that double sum has $(A,D) \in R$ with $A \in X$, hence $D \in R(X)$; so
the sum is at most $\sum_{D \in R(X)}\sum_{A \in P} w(A,D) = p\,|R(X)|$, which is
\eqref{eq:crit}.

\smallskip
\noindent\emph{Sufficiency.} Assume \eqref{eq:crit} and build a network $\mathcal{N}$ as
follows.
\begin{itemize}
\item Vertices: a source $\sigma$, a sink $\tau$, a vertex $u_A$ for each $A \in P$, and a vertex $v_D$ for each $D \in Q$.
\item Arcs $\sigma \to u_A$ of capacity $q$, one for each $A \in P$.
\item Arcs $u_A \to v_D$ of capacity $\infty$, one for each pair $(A,D) \in R$ and no others.
\item Arcs $v_D \to \tau$ of capacity $p$, one for each $D \in Q$.
\end{itemize}

The dictionary between $\mathcal{N}$ and the objects of the theorem is this. The vertices
$u_A$ correspond to the elements of $P$, that is to the rows of $R$, and the vertices $v_D$ to the elements of $Q$, its columns;
the middle arcs are exactly the cells of $R$, so an assignment of values to them is exactly a function vanishing off $R$; the capacity $q$ on $\sigma \to u_A$ is the required row sum of $w$ and the capacity
$p$ on $v_D \to \tau$ its required column sum. Precisely: given a flow $f$ of value $pq$,
put $w(A,D) = f(u_A \to v_D)$ for $(A,D) \in R$ and $w(A,D) = 0$ otherwise. The total
capacity leaving $\sigma$ is $pq$ and the total capacity entering $\tau$ is $pq$, so a flow
of value $pq$ saturates every arc at $\sigma$ and every arc at $\tau$. Conservation at
$u_A$ then reads $\sum_D w(A,D) = q$ and conservation at $v_D$ reads $\sum_A w(A,D) = p$,
which is \eqref{eq:norm}. Conversely a $w$ satisfying \eqref{eq:norm} defines in the same way a flow of value $pq$. So it suffices to produce a flow of value $pq$. 

No flow has value exceeding $pq$, the capacity of the cut $\{\sigma\}$; so the theorem applies, and it is enough to show that every cut has capacity at least $pq$.

Let $S$ be a cut and put $X = \{A \in P : u_A \in S\}$ and $Y = \{D \in Q : v_D \in S\}$.
If some $A \in X$ had a $D \in N(A)$ with $D \notin Y$, the arc $u_A \to v_D$ would run
from $S$ to its complement and the cut would have infinite capacity; so we may assume
$R(X) \subseteq Y$. The arcs leaving $S$ are then precisely the arcs $\sigma \to u_A$ with
$A \notin X$, of total capacity $q\,(p - |X|)$, together with the arcs $v_D \to \tau$ with
$D \in Y$, of total capacity $p\,|Y|$. Hence, using $R(X) \subseteq Y$ and then
\eqref{eq:crit},
\[
c(S) \;=\; q\,\bigl(p - |X|\bigr) + p\,|Y| \;\ge\; pq - q|X| + p\,|R(X)| \;\ge\; pq .
\]
Every cut therefore has capacity at least $pq$, so the maximum flow has value $pq$ and the
required weighting exists.
\end{proof}

\begin{example}\label{ex:fail}
The hypothesis of Theorem~A can fail. Let $P = \{p_1,p_2,p_3\}$, $Q = \{q_1,q_2\}$ and $R = \{(p_1,q_2),(p_2,q_2),(p_3,q_1)\}$, which has no empty row or column and so represents the structure matrix of a $0$-rectangular band (this is essentially \cite[Example 1.3]{H14}). The elements $(p_1,q_1)$ and $(p_2,q_1)$ of $P \times Q$ are each related in $\Gamma(R)$ only to $(p_3,q_2)$, since $M(q_1) = \{p_3\}$ and
$N(p_1) = N(p_2) = \{q_2\}$. Therefore no permutation is contained in $\Gamma(R)$, and the $\mathcal{D}$-class has no permutation matching.  Taking $X = \{p_1,p_2\}$ gives $R(X) = \{q_2\}$ and so
\[
|Q|\,|X| = 2\times 2 =  4 \;>\; 3 = 3 \times 1 = |P|\,|R(X)| ,
\] violating the inequality (1), in accord with Theorem \ref{thm:criterion}. 
\end{example}

\section{Permutation matchings for $O_n$ }
\label{sec:On}

A mapping $\alpha \in \On$ is determined by the pair $(A,B)$, where $A = \ker \alpha$ and $B = \im \alpha$, the kernel classes themselves being intervals. With this in view, let $D_r$ be the set of all members of $\On$ of rank $r$ and for $1 \le r \le n$, let $P_r$ be the set of ordered partitions $A = (A_1 < \cdots < A_r)$ of $[n]$ into $r$ non-empty intervals and $Q_r$ the set of $r$-subsets $B = \{b_1 < \cdots < b_r\}$ of $[n]$. Writing $A_i = [a_{i-1}+1,\,a_i]$ with $a_0 = 0$ and $a_r = n$, a member of $P_r$ may be identified with its set of (right hand) \emph{cut points}, $1 \le a_1 < \cdots < a_{r-1} \le n-1$, so $|P_r| = \binom{n-1}{r-1}$, while $|Q_r| = \binom{n}{r}$. In this way $D_r$ is in bijection with $P_r \times Q_r$, and so $(A,B)\in P_r \times Q_r$ represents a typical member of $D_r$.  

\begin{definition}
For $A \in P_r$ and $D = \{d_1 < \cdots < d_r\} \in Q_r$, we say that $D$ is a \emph{transversal} of $A$, written $A \Rel D$, if $d_i \in A_i$ for every $i$.
\end{definition}

It is readily checked  that $\mathcal{R}$- and $\mathcal{L}$-classes of $\On$ correspond to equality of kernels and images respectively, $\mathcal{H}$ is the equality relation, (i.e. $\On$ is combinatorial), while $\mathcal{D} = \mathcal{J}$ corresponds to equality of rank, and so $D_r$ is a typical $\mathcal{D}$-class of $\On$. Moreover, $\alpha = \alpha^2$ if and only if $\ker \alpha \Rel \im \alpha$, from which it follows that $\On$ is regular. 

Observe that $A \Rel D$ (noting that $R$ here does \emph{not} denote the right Green's relation) holds precisely when $|D \cap A_i| = 1$ for every $i$, and since the $A_i$ form an increasing sequence of intervals, the unique member of $D$ in $A_i$ is therefore the $i$-th smallest member of $D$. 
\begin{remark}{\label{rem:trans}}
The structure matrix $R$ of $D_r$ is defined by the transversal relation in that $(A,B)$ represents an idempotent exactly when $A \Rel B$; this is because a map of rank $r$ is idempotent, and so generates a $1$ at position $(A,B)$, precisely when it fixes each of the $r$ members of its image, that is when $b_i \in A_i$ for every $i$.
\end{remark}
 
\section*{The deletion incidence}\label{sec:del}

Fix $1 \le r \le n$, put $\Sigma = \{0,1,\dots,n-r\}$ and $m = |\Sigma| = n-r+1$; for $k \ge 0$ let $M_k$ be the set of weakly increasing $k$-tuples over $\Sigma$, equivalently of multisets $s$ of size $k$ drawn from $\Sigma$, so $|M_k| = \binom{m+k-1}{k}$ and $M_0$ is a single empty tuple. 
\footnote{$|M_k|$ counts solutions, in non-negative integers, of $x_1 + \cdots + x_m = k$, solved by the \emph{stars-and-bars} method.}

Write $|s|_v$ for the multiplicity of $v$ in $s$, and $s \uplus \{v\}$ for $s$ with one further copy of $v$ adjoined (whether or not it already has one). 

\begin{lemma}\label{lem:enc}
Let $A \in P_r$ have cut points $a_1 < \cdots < a_{r-1}$ and let $D = \{d_1<\cdots<d_r\} \in Q_r$. Put
\[
s(A) = \bigl(a_1-1,\ \dots,\ a_{r-1}-(r-1)\bigr), \qquad t(D) = \bigl(d_1-1,\ \dots,\ d_r-r\bigr).
\]
Then $s : P_r \to M_{r-1}$ and $t : Q_r \to M_r$ are bijections, with $|M_{r-1}| = \binom{n-1}{r-1}$ and $|M_r| = \binom{n}{r}$; and, writing $s(A) = (s_1,\dots,s_{r-1})$ and $t(D) = (t_1,\dots,t_r)$,
\[
A \Rel D \qquad\Longleftrightarrow\qquad t_1 \le s_1 \le t_2 \le s_2 \le \cdots \le s_{r-1} \le t_r .
\]
\end{lemma}

\begin{proof}
The cut points of members of $P_r$ are exactly the sequences $1 \le a_1 < \cdots < a_{r-1} \le n-1$. Subtracting $i$ from $a_i$ converts strict increase into weak increase and is invertible, the bounds becoming $0 \le s_1$ and $s_{r-1} \le (n-1)-(r-1) = n-r$; so $s$ is a bijection onto $M_{r-1}$, of size $\binom{m+r-2}{r-1} = \binom{n-1}{r-1}$ as $m+r-2 = n-1$. In the same way $1 \le d_1 < \cdots < d_r \le n$ corresponds to $0 \le t_1 \le \cdots \le t_r \le n-r$, and $|M_r| = \binom{m+r-1}{r} = \binom{n}{r}$.

For the last part, $A \Rel D$ says $a_{i-1} < d_i \le a_i$ for $i = 1,\dots,r$. The extreme constraints $0 < d_1$ and $d_r \le n$ hold automatically. For $1 \le i \le r-1$, the condition $d_i \le a_i$ reads $t_i + i \le s_i + i$, that is $t_i \le s_i$; and $a_i < d_{i+1}$ reads $s_i + i < t_{i+1}+i+1$, that is $s_i \le t_{i+1}$. Together these give the displayed chain.
\end{proof}

\begin{definition}
The \emph{deletion incidence} is $e : M_{r-1}\times M_r \to \mathbb{Z}_{\ge0}$,
\[
e(s,t) \;=\; \#\bigl\{\, i \in \{1,\dots,r\} \;:\; t \text{ with its } i\text{-th entry deleted equals } s \,\bigr\} .
\]
\end{definition}

\begin{lemma}\label{lem:del}
For all $s \in M_{r-1}$ and $t \in M_r$:
\begin{enumerate}
\item if $e(s,t)>0$ then $s$ and $t$ interlace, in the sense of Lemma~\ref{lem:enc};
\item $\sum_{s \in M_{r-1}} e(s,t) = r$ for every $t \in M_r$;
\item $\sum_{t \in M_r} e(s,t) = n$ for every $s \in M_{r-1}$.
\end{enumerate}
\end{lemma}

\begin{proof}
(1) Write $t = (t_1 \le \cdots \le t_r)$ and let $s = (s_1 \le \cdots \le s_{r-1})$ arise by deleting the $i$-th entry, so $s_j = t_j$ for $j<i$ and $s_j = t_{j+1}$ for $j \ge i$. We check $t_j \le s_j \le t_{j+1}$ for $1 \le j \le r-1$. If $j<i$ then $t_j \le s_j$ holds with equality, while $s_j = t_j \le t_{j+1}$ as $t$ is weakly increasing. If $j \ge i$ then $t_j \le t_{j+1} = s_j$ as $t$ is weakly increasing, while $s_j \le t_{j+1}$ holds with equality.

(2) Deleting any one of the $r$ entries of $t$ leaves some member of $M_{r-1}$, so the sum counts the positions of $t$.

(3) Fix $s$, so $\sum_{v\in\Sigma} |s|_v = r-1$. A multiset $t$ contributes only if deleting some entry of $t$ leaves $s$, hence only if $t = s \uplus \{v\}$ for some $v \in \Sigma$; distinct $v$ give distinct $t$, so these contributions do not overlap. For $t = s\uplus\{v\}$ the positions whose deletion leaves $s$ are exactly those carrying the value $v$, of which there are $|s|_v + 1$. Hence
\[
\sum_{t \in M_r} e(s,t) \;=\; \sum_{v\in\Sigma}\bigl(|s|_v  + 1\bigr) \;=\; (r-1)+|\Sigma| \;=\; (r-1)+(n-r+1) \;=\; n ,
\]
which does not depend on $s$.
\end{proof}

As a check on the degrees, $r|M_r| = r\binom nr = n\binom{n-1}{r-1} = n|M_{r-1}|$, as bi-regularity requires. It is the multiplicity that makes the count in (3) constant: a repeated value of $s$ is reachable by more deletions, compensated exactly because $\sum_{v \in \Sigma} \bigl(|s|_v + 1\bigr) =
(r-1) + |\Sigma|$ does not depend on $s$. The underlying $0$-$1$ relation, interlacing itself, is in general not bi-regular, as seen in the next example, which is why the incidence must be counted with multiplicity. 

\begin{example}
    
When $n=3$, $r=2$ we have $\Sigma=\{0,1\}$, and $e$ is the following
array, with rows indexed by $M_{r-1} = \{(0),(1)\}$ and columns by
$M_r = \{(0,0),(0,1),(1,1)\}$, margins shown:
\[
\begin{array}{c|ccc|c}
      & (0,0) & (0,1) & (1,1) & \\\hline
(0)   & 2     & 1     & 0     & 3\\
(1)   & 0     & 1     & 2     & 3\\\hline
      & 2     & 2     & 2     &
\end{array}
\]
The two zeros in the array indicate where pairs fail to interlace: deleting an entry of $(1,1)$ yields $(1)$, never $(0)$, and similarly deleting from $(0,0)$ cannot yield a $(1)$.

\end{example}

\begin{corollary}\label{cor:weighting}
For each $r$ the transversal relation $R \subseteq P_r \times Q_r$ admits the balanced weighting $w_r(A,D) = e\bigl(s(A),t(D)\bigr)$, with constants $a = n$ and $b = r$.
\end{corollary}

\begin{proof}
By Lemma~\ref{lem:del}(1) with Lemma~\ref{lem:enc}, $w_r$ vanishes off $R$; by Lemma~\ref{lem:del}(3) each row sum is $n$ and by Lemma~\ref{lem:del}(2) each column sum is $r$, the bijections of Lemma~\ref{lem:enc} matching the index sets.
\end{proof}

\section*{Proof of Theorem B} 
\begin{proof}
By Corollary 3.6 the structure matrix of $D_r$, which by Definition 3.1 and Remark~\ref{rem:trans} is the transversal relation $R \subseteq P_r \times Q_r$,
admits a balanced weighting. Hence by Theorem~A each $D_r$ has a permutation
matching, and since the inverses of any element lie in its own
$\mathcal{D}$-class, these combine to give a permutation matching of
$\mathcal{O}_n$.
\end{proof}

\medskip

ACKNOWLEDGEMENT The paper was written in conjunction with the LLM Claude (Anthropic), Opus 5/5.5.

\end{document}